\documentclass[11pt]{amsart}
\usepackage{graphicx}
\usepackage[centertags]{amsmath}
\usepackage{amsfonts}
\usepackage{amssymb}
\usepackage{amsthm}
\usepackage{newlfont}
\newtheorem{thm}{Theorem}[section]
\newtheorem{cor}[thm]{Corollary}

\theoremstyle{definition}

\theoremstyle{remark}

\title{On affine connections in a Riemannian manifold with a circulant metric and two circulant affinor structures$^{1}$}
\bigskip
\author{Iva Dokuzova, Dimitar Razpopov}
\date{}
\begin{document}
\footnotetext[1]{This work is partially supported by project RS09 - FMI - 003 of the Scientific Research Fund, Paisii Hilendarski University of Plovdiv, Bulgaria}
\maketitle

\begin{abstract}
In the present paper it is considered a class $V$ of
$3$-dimensional Riemannian manifolds $M$ with a metric $g$ and two
affinor tensors $q$ and $S$. It is defined another metric
$\bar{g}$ in $M$. The local coordinates of all these tensors are
circulant matrices. It is found:

1)\  a relation between curvature tensors $R$ and
$\bar{R}$ of $g$ and $\bar{g}$, respectively;

2)\   an identity of the curvature tensor $R$ of $g$ in
the case when the curvature tensor $\bar{R}$ vanishes;

3)\  a relation between the sectional curvature of a $2$-section of the type $\{x, qx\}$ and the scalar curvature of $M$.
\end{abstract}

\section{Preliminaries}
\thispagestyle{empty}

We consider a $3$-dimensional Riemannian manifold $M$ with a
metric tensor $g$ and two affinor tensors $q$ and $S$ such that:
their local coordinates form circulant matrices. So these matrices
are as follows:
\begin{equation}\label{f1}
    g_{ij}=\begin{pmatrix}
      A & B & B \\
      B & A & B \\
      B & B & A \\
    \end{pmatrix}, \quad A > B > 0,
\end{equation}
where $A$ and $B$ are smooth functions of a point $p(x^{1}, x^{2},
x^{3})$ in some $F\subset R^{3}$,
\begin{equation}\label{f2}
    q_{i}^{.j}=\begin{pmatrix}
      0 & 1 & 0 \\
      0 & 0 & 1 \\
      1 & 0 & 0 \\
    \end{pmatrix},\qquad S_{i}^{j}=\begin{pmatrix}
      -1 & 1 & 1 \\
      1 & -1 & 1 \\
      1 & 1 & -1 \\
    \end{pmatrix}.
\end{equation}

Let $\nabla$ be the connection of $g$. The following results have
been obtained in \cite{1}.
\begin{equation}\label{f3}
    q^{3}=E;\quad g(qu, qv)=g(u,v),\quad u,\ v\in \chi M.
\end{equation}
\begin{equation}\label{f4}
    \nabla q =0 \quad \Leftrightarrow \quad grad A=grad B . S .
\end{equation}
\begin{equation}\label{f5}
    0 < B < A \quad \Rightarrow \quad g \ is \ possitively \ defined.
\end{equation}

If $M$ has a metric $g$ from (\ref{f1}), affinor structures $q$
and $S$ from (\ref{f2}) and $\nabla q=0$, then we note for brevity,
that $M$ is in the class $V$.

We denote $\tilde{q}_{j}^{s}=q_{a}^{s}q_{j}^{a}, \ \Phi_{j}^{s}=
q_{j}^{s}+\tilde{q}_{j}^{s}$, and from (\ref{f2}) we have:
\begin{equation}\label{f6}
    \tilde{q}_{j}^{.s}=\begin{pmatrix}
      0 & 0 & 1 \\
      1 & 0 & 0 \\
      0 & 1 & 0 \\
    \end{pmatrix},\qquad \Phi_{j}^{s}=\begin{pmatrix}
      0 & 1 & 1 \\
      1 & 0 & 1 \\
      1 & 1 & 0 \\
    \end{pmatrix}.
\end{equation}


\section{On affine connections in a Riemannian manifold }

Let $M$ be in $V$.  We note
$f_{ij}=g_{ik}q_{j}^{k}+g_{jk}q_{i}^{k}$, i.e.
\begin{equation}\label{f7}
f_{ij}=\begin{pmatrix}
      2B & A+B & A+B \\
      A+B & 2B & A+B \\
      A+B & A+B & 2B \\
    \end{pmatrix}.
\end{equation}
We calculate $\det f_{ij}=2(A-B)^{2}(A+2B)\neq 0$, so $f$ is
non-degenerated symmetric tensor field. Evidently, we have that
$\nabla q =0$, which thank's to (\ref{f2}), (\ref{f6}),
(\ref{f7}), implies:
\begin{equation}\label{f8}
     \nabla \tilde{q} =0, \qquad \nabla
    f=0, \qquad \nabla S=0, \qquad \nabla\Phi=0.
\end{equation}
For later use, from (\ref{f1}) -- (\ref{f7}), we find the next
identities:
\begin{equation}\label{f9}
    \Phi_{j}^{s}g_{is}=f_{ji}, \qquad
    \Phi_{j}^{s}f_{is}=2g_{ji}+f_{ji}, \qquad
    f_{ji}g^{is}=\Phi_{j}^{s}\qquad
    g_{ji}f^{is}=\frac{1}{2}S_{j}^{s}.
\end{equation}

Further, we suppose $\alpha$ and $\beta$ are two smooth functions
in $F$. Now, we construct the metric $\bar{g}$, as follows
\begin{equation}\label{f10}
    \bar{g}=\alpha .g +\beta .f.
\end{equation}
The local coordinates of $\bar{g}$ are
\begin{equation*}
\bar{g}_{ij}=\begin{pmatrix}
      \alpha A+2\beta B & \beta A+(\alpha+\beta)B & \beta A+(\alpha+\beta)B  \\
      \beta A+(\alpha+\beta)B  & \alpha A+2\beta B & \beta A+(\alpha+\beta)B \\
      \beta A+(\alpha+\beta)B  & \beta A+(\alpha+\beta)B & \alpha A+2\beta B \\
    \end{pmatrix}.
\end{equation*}
In \cite{2} it is proved the next assertion.
\begin{thm}
 Let $M$ be a manifold in $V$, also  $g$ and $\bar{g}$ be two
metrics of $M$, related by (\ref{f10}). Let $\nabla$ and
$\overline{\nabla}$ be the corresponding connections of $g$ and
$\bar{g}$. Then $\overline{\nabla}q=0$ if and only if, when
\begin{equation}\label{statiq5}
    grad \alpha=grad \beta . S.
\end{equation}
\end{thm}

Let $\beta =0$ in (\ref{f10}). Then we have
\begin{equation}\label{f101}
    \bar{g}=\alpha .g.
\end{equation}
The condition (\ref{f101}) defines the well known conformal transformation in the Riemannian manifold $M$.

So, we will consider the case
$\alpha=0$ in (\ref{f10}). We obtain
\begin{equation}\label{f10*}
    \bar{g}_{ij}=\beta .f_{ij}.
\end{equation}
Then from (\ref{statiq5}) we can get that $\overline{\nabla}q=0$
if and only if, when
$\beta$ is a constant.
Further, we are interested in the case $\overline{\nabla}q\neq 0$, i.e. $\beta$ isn't a constant.
Thank's to (\ref{f10*}) we get
\begin{equation}\label{f14}
\nabla_{k}\bar{g}_{ij}=\beta_{k}f_{ij},\qquad
\beta_{k}=\nabla_{k}\beta .
\end{equation}
We have the well known identities:
\begin{equation}\label{f11}
\overline{\nabla}_{k}\bar{g}_{ij}=\partial_{k}\bar{g}_{ij}-\overline{\Gamma}_{ki}^{a}\bar{g}_{aj}-\overline{\Gamma}_{kj}^{a}\bar{g}_{ai},
\end{equation}
\begin{equation}\label{f12}
\nabla_{k}\bar{g}_{ij}=\partial_{k}\bar{g}_{ij}-\Gamma_{ki}^{a}\bar{g}_{aj}-\Gamma_{kj}^{a}\bar{g}_{ai},
\end{equation}
\begin{equation}\label{f13}
\overline{\nabla}_{k}\bar{g}_{ij}=0.
\end{equation}
Using (\ref{f9}), (\ref{f14}) -- (\ref{f13}), for the tensor
$T_{ik}^{s}=\overline{\Gamma}_{ki}^{s}-\Gamma_{ki}^{s}$ of the affine deformation of $\nabla$ and $\overline{\nabla}$,
we find
\begin{equation}\label{f15}
T_{ik}^{s}=\beta_{k}\delta_{i}^{s}+\beta_{i}\delta_{k}^{s}-\frac{1}{2}\beta^{a}S_{a}^{s}f_{ik},
\qquad \beta_{k}\sim \frac{\beta_{i}}{2\beta}.
\end{equation}
We have that $\overline{\nabla}_{i}q_{j}^{k}=\nabla_{i}
q_{j}^{k}-\beta_{j}q_{i}^{k}+\tilde{\beta}_{j}\delta_{i}^{k}-\dfrac{1}{2}\beta^{a}S_{a}^{k}q_{j}^{t}f_{ti}+\dfrac{1}{2}\beta^{a}S_{a}^{t}q_{t}^{k}f_{ij}$.

Let $R$ be the curvature tensor field of $\nabla$. Let $\bar{R}$
be the curvature tensor field of $\overline{\nabla}$. It is well
known the relation (see \cite{3})
\begin{equation}\label{f16}
 \overline{R}^{h}_{ijk} = R^{h}_{ijk} + \nabla_{j}T^{h}_{ik}-\nabla_{k}T^{h}_{ij}
+T^{s}_{ik}T^{h}_{sj}-T^{s}_{ij}T^{h}_{sk}.
\end{equation}
From (\ref{f15}) and (\ref{f16}) after some calculations we
obtain
\begin{equation}\label{f17}
\begin{split}
     \overline{R}^{h}_{ijk}& = R^{h}_{ijk} + \delta_{k}^{h}(\nabla_{j}\beta_{i}-\beta_{i}\beta_{j}+\varphi f_{ij})-\delta_{j}^{h}(\nabla_{k}\beta_{i}-\beta_{i}\beta_{k}+\varphi f_{ik})\\
 &+\frac{1}{2}f_{ij}S_{t}^{h}(\nabla_{k}\beta^{t}-\beta_{k}\beta^{t})-\frac{1}{2}f_{ik}S_{t}^{h}(\nabla_{j}\beta^{t}-\beta_{j}\beta^{t}), \quad \varphi= \frac{1}{2}\beta^{t}\beta_{s}S_{t}^{s}.
\end{split}
\end{equation}
\begin{thm}\label{th2.4}
    Let $M$ be in $V$, also $\nabla$ and $\overline{\nabla}$ be the Riemannian connections of $g$ and $\bar{g}$, related by (\ref{f10*}). If $\overline{\nabla}$
    is a locally flat connection, then
    curvature tensor field $R$ of $\nabla$ is
 \begin{equation}\label{f18}
\begin{split}
R(x, y, z, u)& = \frac{\tau}{6}[(2g(x, u)g(y, z)-2g(x,z)g(y,u)\\
&\mbox{}+(g(qx, u)+g(x, qu))(g(qy, z)+ g(y, qz))\\
\ &\mbox{}-(g(qx,z)+g(x,qz))(g(qy,u)+g(y,qu))].
\end{split}
\end{equation}
\end{thm}
\begin{proof}
We have $\overline{R}=0$. From (\ref{f17}) we find
\begin{equation}\label{f19}
     R^{h}_{ijk} = \delta_{j}^{h}P_{ki}-\delta_{k}^{h}P_{ij}
 -f_{ij}Q_{k}^{h}+f_{ik}Q_{j}^{h},
\end{equation}
where $P_{ki}=\nabla_{k}\beta_{i}-\beta_{i}\beta_{k}+\varphi
f_{ik}$,\
$Q_{k}^{h}=\dfrac{1}{2}S_{t}^{h}(\nabla_{k}\beta^{t}-\beta_{k}\beta^{t})$.

Now, we contract $k=h$ in (\ref{f19}) and with the help of
(\ref{f9}) we get
\begin{equation}\label{f20}
    R_{ij}=-P_{ij}-\psi f_{ij}, \qquad
    \psi=\frac{1}{2}S_{t}^{h}\nabla_{h}\beta^{t}.
\end{equation}
We note that $R_{ij}=R^{k}_{ijk}$ are the local coordinates of the
Ricci tensor of $\nabla$, also $\tau=R_{ij}g^{ij}$ and
$\tau^{*}=R_{ij}f^{ij}$ are the first and the second scalar
curvatures of $M$, respectively. The identity (\ref{f20}) implies
\begin{equation}\label{f21}
    \tau^{*}=-2\varphi-4\psi.
\end{equation}
Using (\ref{f9}) we have that $Q_{k}^{h}=P_{ka}f^{ah}-\varphi
\delta_{k}^{h}$, and from (\ref{f20}) we get
\begin{equation}\label{f22}
    Q_{k}^{h}=-R_{ka}f^{ah}-(\psi+\varphi) \delta_{k}^{h}.
\end{equation}
We substitute (\ref{f20}) -- (\ref{f22}) in (\ref{f19}), and we find
\begin{equation}\label{f23}
     R^{h}_{ijk} = \delta_{k}^{h}(R_{ij}-\frac{\tau^{*}}{2}f_{ij})-\delta_{j}^{h}(R_{ki}-\frac{\tau^{*}}{2}f_{ki})
 +f_{ij}R_{ka}f^{ah}-f_{ik}R_{ja}f^{ah}.
\end{equation}
From (\ref{f23}) and  $R_{k}^{h}=R_{ijk}^{h}g^{ij}$ we have
\begin{equation}\label{f24}
    2R_{k}^{h}=\tau\delta_{k}^{h}+\frac{\tau^{*}}{2}\Phi_{k}^{h}-\Phi_{k}^{t}R_{ta}f^{ah}.
\end{equation}

Now, we contract (\ref{f24}) with $f_{ih}$, and from identity
$f_{ih}R_{k}^{h}= \Phi_{i}^{a}R_{ka}$ we obtain:
\begin{equation*}
    2\Phi_{i}^{a}R_{ka}=(\frac{\tau^{*}}{2}+\tau)f_{ki}+\tau^{*}g_{ki}-\Phi_{k}^{t}R_{ti}
\end{equation*}
and
\begin{equation*}
    2\Phi_{k}^{a}R_{ia}=(\frac{\tau^{*}}{2}+\tau)f_{ki}+\tau^{*}g_{ki}-\Phi_{i}^{t}R_{tk}.
\end{equation*}
The last system of two equations implies
\begin{equation}\label{f27}
    \Phi_{k}^{a}R_{ia}=\frac{1}{3}((\frac{\tau^{*}}{2}+\tau)f_{ki}+\tau^{*}g_{ki}).
\end{equation}
From (\ref{f9}) and (\ref{f27}) we find
\begin{equation}\label{f28}
    \Phi_{k}^{a}R_{ia}f^{ij}=\frac{1}{3}((\frac{\tau^{*}}{2}+\tau)\delta_{k}^{j}+\tau^{*}S_{k}^{j}).
\end{equation}
After substituting (\ref{f28}) in (\ref{f24}), we get
\begin{equation*}
    R_{k}^{h}=\frac{\tau}{3}\delta_{k}^{h}+\frac{\tau^{*}}{6}\Phi_{k}^{h},
\end{equation*}
and also
\begin{equation}\label{f30}
    R_{ki}=\frac{\tau}{3}g_{ki}+\frac{\tau^{*}}{6}f_{ki},\quad R_{ki}f^{ih}=\frac{\tau}{6}S_{k}^{h}+\frac{\tau^{*}}{6}\delta_{k}^{h}.
\end{equation}
From the last equation we find that $$\tau^{*}=-\tau.$$
The equations (\ref{f30}) get the form
\begin{equation}\label{f33}
    R_{ki}=\frac{\tau}{6}(2g_{ki}-f_{ki}),\quad R_{ki}f^{ih}=\frac{\tau}{6}(S_{k}^{h}-\delta_{k}^{h}).
\end{equation}
Finely we obtain:
\begin{equation*}
     R^{h}_{ijk} =
     \frac{\tau}{6}(2\delta_{k}^{h}g_{ij}-2\delta_{j}^{h}g_{ki}+(\delta_{k}^{h}+S_{k}^{h})f_{ij}-(\delta_{j}^{h}+S_{j}^{h})f_{ki})
\end{equation*}
and
\begin{equation*}
     R_{hijk} = \frac{\tau}{6}(2g_{kh}g_{ij}-2g_{hj}g_{ki}+f_{kh}f_{ij}-f_{hj}f_{ki}).
\end{equation*}
The last identity is equivalent to (\ref{f18}).

We note that $R_{ijk}^{h}\neq 0$, so $\nabla$
    isn't a locally flat connection.
\end{proof}
Let $p$ be a point in $M$ and $x$, $y$ be two linearly
independent vectors in $T_{p}M$. It is known the value
\begin{equation*}
    \mu(x,y)=\frac{R(x, y, x, y)}{g(x, x)g(y, y)-g^{2}(x, y)}
\end{equation*}
is the sectional curvature of $2$-section $\{x, y\}$.
\begin{cor}
Let $M$ satisfy the conditions of the Theorem \ref{th2.4}. Let $x$
be an arbitrary vector in $T_{p}M$, and $\varphi$ be the angle
between $x$ and $qx$. Then the sectional curvature of a
$2$-section $\{x, qx\}$ is
\begin{equation*}
    \mu(x,qx)=-\frac{\tau}{6}\tan^{2}\frac{\varphi}{2}, \qquad\varphi\in
    (0,\dfrac{2\pi}{3}).
\end{equation*}
\end{cor}
\begin{cor}
Let $M$ satisfy the conditions of the Theorem \ref{th2.4}. Then
the Ricci tensor of $g$ is degenerated.
\end{cor}
The proof follows from (\ref{f33}).

Note. Let $\{x, qx\}$  be a $2$-section and $g(x, qx)=0$. Then
\begin{equation*}
    \mu(x,qx)=-\frac{\tau}{6}.
\end{equation*}

\vspace{6mm}
\author{Iva Dokuzova \\University of Plovdiv\\ FMI,  Department of
Geometry\\236 Bulgaria Blvd.\\Bulgaria 4003\\\vspace{6mm}
e-mail:dokuzova@uni-plovdiv.bg}\\
\author{Dimitar Razpopov \\ Department of Mathematics and Physics\\ Agricultural University of Plovdiv \\12 Mendeleev Blvd.\\
Bulgaria 4000 \\
e-mail:drazpopov@qustyle.bg}

\vspace{6mm} \textbf{Mathematics Subject Classification (2010)}:
53C15, 53B20

\vspace{4mm} \textbf{Keywords}: Riemannian manifold, affinor
structure, curvatures

\end{document}